\documentclass[english]{elsarticle}
\usepackage[T1]{fontenc}
\usepackage[latin9]{inputenc}
\usepackage{geometry}
\usepackage{xcolor}
\usepackage{pdfcolmk}
\usepackage{babel}
\usepackage{amstext}
\usepackage{amsthm}
\usepackage{amssymb}
\usepackage{graphicx}
\PassOptionsToPackage{normalem}{ulem}
\usepackage{ulem}
\usepackage[unicode=true,
 bookmarks=true,bookmarksnumbered=false,bookmarksopen=false,
 breaklinks=false,pdfborder={0 0 1},backref=false,colorlinks=false]
 {hyperref}

\makeatletter

\providecommand{\tabularnewline}{\\}
\providecolor{lyxadded}{rgb}{0,0,1}
\providecolor{lyxdeleted}{rgb}{1,0,0}

\DeclareRobustCommand{\lyxsout}[1]{\ifx\\#1\else\sout{#1}\fi}

\theoremstyle{plain}
\newtheorem{thm}{\protect\theoremname}[section]

\makeatletter
\let\eqref\relax
\usepackage{amsmath}
\usepackage{mathtools}
\usepackage{bbm}

\@ifundefined{showcaptionsetup}{}{%
 \PassOptionsToPackage{caption=false}{subfig}}
\usepackage{subfig}
\makeatother

\providecommand{\theoremname}{Theorem}

\begin{document}

\begin{frontmatter}{}

\title{On the existence of non-trivial steady-state size-distributions for
a class of flocculation equations}

\author{Inom Mirzaev}

\ead{mirzaev.1@mbi.osu.edu}

\address{Mathematical Biosciences Institute, Ohio State University, Columbus,
OH, United States}

\author{David M. Bortz}

\ead{dmbortz@colorado.edu}

\address{Department of Applied Mathematics, University of Colorado, Boulder,
CO, United States}
\begin{abstract}
Flocculation is the process whereby particles (i.e., \emph{flocs})
in suspension reversibly combine and separate. The process is widespread
in soft matter and aerosol physics as well as environmental science
and engineering. We consider a general size-structured flocculation
model, which describes the evolution of floc size distribution in
an aqueous environment. Our work provides a unified treatment for
many size-structured models in the environmental, industrial, medical,
and marine engineering literature. In particular, the mathematical
model considered in this work accounts for basic biological phenomena
in a population of microorganisms including growth, death, sedimentation,
predation, surface erosion, renewal, fragmentation and aggregation.
The central objective of this work is to prove existence of positive
steady states of this generalized flocculation model. Using results
from fixed point theory we derive conditions for the existence of
continuous, non-trivial stationary solutions. We further develop a
numerical scheme based on spectral collocation method to approximate
these positive stationary solutions. We explore the stationary solutions
of the model for various biologically relevant parameters and give
valuable insights for the efficient removal of suspended particles.
\end{abstract}

\end{frontmatter}{}

\section{Introduction}

\emph{Flocculation} is the process whereby particles (i.e., flocs)
in suspension reversibly combine and separate. The process is widespread
in soft matter and aerosol physics as well as environmental science
and engineering. For instance, in fields such as water treatment,
biofuel production and beer fermentation, flocculation process is
often used to enhance suspended solids removal. One of the most important
design and control parameters for the efficient removal of suspended
particles is the size distribution of the flocs in stirring tanks.
A popular mathematical model describing the time-evolution of the
particle size distribution in a stirring tank is a 1D nonlinear partial
integro-differential equation based on the population-balance equations
proposed by \citet{Smoluchowski1917}. The model have been successful
in matching many flocculation experiments \citep{li2004modelling,ducoste2002atwoscale,spicer1996coagulation}. 

Previous analytical work on these models focused on classes of flocculation
equations that did not allow for the \emph{vital dynamics} (i.e.,
birth and death) of individual particles. These phenomena are obviously
critical features in the modeling of microbial flocculation. Accordingly,
in this work we consider a general size-structured flocculation model
which accounts for growth, aggregation, fragmentation, surface erosion
and sedimentation. The equations for the microbial flocculation model
track the time-evolution of the particle size number density $u(t,\,x)$
and can be written as 
\begin{eqnarray}
u_{t} & = & \mathcal{F}(u)\label{eq: agg and growth model}
\end{eqnarray}
where
\[
\mathcal{F}(u):=\mathcal{G}(u)+\mathcal{A}(u)+\mathcal{B}(u),
\]
$\mathcal{G}$ denotes growth
\begin{equation}
\mathcal{G}(u):=-\partial_{x}(gu)-\mu(x)u(t,\,x)\,,\label{eq:Growth}
\end{equation}
$\mathcal{A}$ denotes aggregation
\begin{alignat}{1}
\mathcal{A}(u) & :=\frac{1}{2}\int_{0}^{x}k_{a}(x-y,\,y)u(t,\,x-y)u(t,\,y)\,dy\nonumber \\
 & \quad-u(t,\,x)\int_{0}^{\overline{x}-x}k_{a}(x,\,y)u(t,\,y)\,dy\,,\label{eq:Aggregation}
\end{alignat}
and $\mathcal{B}$ denotes breakage 
\begin{equation}
\mathcal{B}(u):=\int_{x}^{\overline{x}}\Gamma(x;\,y)k_{f}(y)u(t,\,y)\,dy-\frac{1}{2}k_{f}(x)u(t,\,x)\,.\label{eq:Breakage}
\end{equation}
The boundary condition is traditionally defined at the smallest size
$0$ and the initial condition is defined at $t=0$

\[
g(0)u(t,\,0)=\int_{0}^{\overline{x}}q(x)u(t,\,x)dx,\quad u(0,\,x)=u_{0}(x)\,,
\]
where the renewal rate $q(x)$ represents the number of new flocs
entering the population. We note that this boundary condition could
also be used to model the surface erosion of flocs, where single cells
are eroded off the floc and enter single cell population. A floc size
is usually expressed as volume. Moreover, since the equations model
flocculation of particles in a confined space, the flocs are assumed
to have a maximum size $\overline{x}<\infty$. The function $g(x)$
represents the average growth rate of the flocs of size $x$ due to
proliferation, and the coefficient $\mu(x)$ represents a size-dependent
removal rate due to gravitational sedimentation and death. 

The aggregation of flocs into larger ones is modeled in (\ref{eq:Aggregation}),
by the Smoluchowski coagulation equation. The function $k_{a}(x,\,y)$
is the aggregation kernel, which describes the rate with which the
flocs of size $x$ and $y$ agglomerate to form a floc of size $x+y$.
This equation has been widely used, e.g., to model the formation of
clouds and smog in meteorology \citep{pruppacher2012microphysics}, the kinetics
of polymerization in biochemistry \citep{ziff1980kinetics}, the clustering
of planets, stars and galaxies in astrophysics \citep{makino1998onthe},
and even schooling of fish in marine sciences \citep{niwa1998schoolsize}. The
equation has also been the focus of considerable mathematical analysis.
For the aggregation kernels satisfying the inequality $k_{a}(x,\,y)\le1+x+y$,
existence of mass conserving global in time solutions were proven
\citep{Dubovski1996,Fournier2005,Menon2005} (for some suitable initial
data). Conversely, for aggregation kernels satisfying $(xy)^{\gamma/2}\le k_{a}(x,\,y)$
with $1<\gamma\le2$, it has been shown that the total mass of the
system blows up in a finite time (referred as a \emph{gelation time})
\citep{Escobedo2002a}. For a review of further mathematical results,
we refer readers to review articles by Aldous \citep{Aldous1999},
Menon and Pego \citep{Menon2006}, and Wattis \citep{Wattis2006a}
and the book by Dubovskii \citep{Dubovskii1994}. Lastly, although
the Smoluchowski equation has received substantial theoretical work,
the derivation of analytical solutions for many realistic aggregation
kernels has proven elusive. Towards this end, many discretization
schemes for numerical simulations of the Smoluchowski equations have
been proposed, and we refer interested readers to the review by Bortz
\citep[\S 6]{bortz2015chapter}.

The breakage of flocs due to fragmentation is modeled by the terms
in (\ref{eq:Breakage}), where the fragmentation kernel $k_{f}(x)$
calculates the rate with which a floc of size $x$ fragments. The
breakage process assumes the fragmentation of a floc of size $x$
into sizes $\{x-y,\,y\}$ and $\{y,\,x-y\}$ as two separate events.
Therefore, the factor of $1/2$ is included in the second sum to avoid
double-counting. The integrable function $\Gamma(x;y)$ represents
the post-fragmentation probability density of daughter flocs for the
fragmentation of the parent flocs of size $y$. The post-fragmentation
probability density function $\Gamma$ is one of the least understood
terms in the flocculation model. Many different forms are used in
the literature, among which normal and log-normal densities are the
most common \citep{Spicer1996a}. Recent modeling and computational
work suggests that normal and log-normal forms for $\Gamma$ are not
correct and that a form closer to an $\arcsin(x;y)$ density would
be more accurate \citep{mirzaev2016aninverse,byrne2011postfragmentation}.
However, in this work we do not restrict ourselves to any particular
form of $\Gamma$, and instead simply assume that the function $\Gamma$
satisfies the mass conservation requirement. In other words, all the
fractions of daughter flocs formed upon the fragmentation of a parent
floc sum to unity,
\begin{equation}
\int_{0}^{y}\Gamma(x;\,y)\,dx=1\text{ for all }y\in(0,\,\overline{x}].\label{eq:number conservation for gamma}
\end{equation}

The microbial flocculation equation, presented in (\ref{eq: agg and growth model}),
is a generalization of many mathematical models appearing in the size-structured
population modeling literature and has been widely used, e.g., to
model the formation of clouds and smog in meteorology \citep{pruppacher2012microphysics},
the kinetics of polymerization in biochemistry \citep{ziff1980kinetics},
the clustering of planets, stars and galaxies in astrophysics \citep{makino1998onthe},
and even schooling of fish in marine sciences \citep{niwa1998schoolsize}.
For example, when the fragmentation kernel is omitted, $k_{f}\equiv0$,
the flocculation model reduces to algal aggregation model used to
describe the evolution of a phytoplankton community \citep{ackleh1997modeling}.
When the removal and renewal rates are set to zero, the flocculation
model simplifies to a model used to describe the proliferation of
\emph{Klebsiella pneumonia }in a bloodstream \citep{bortz2008klebsiella}.
Furthermore, the flocculation model, with only growth and fragmentation
terms, was used to investigate the elongation of prion polymers in
infected cells \citep{calvez2012selfsimilarity,doumic-jauffret2009eigenelements,calvez2010priondynamics}. 

Investigating asymptotic behavior of the equation (\ref{eq: agg and growth model})
has been a challenging task because of the nonlinearity introduced
by the aggregation terms. Nevertheless, under suitable conditions
on the kernels, the existence of a positive steady state has been
established for the pure aggregation and fragmentation case \citep{laurencot2005steadystates}.
For the case $\overline{x}=\infty$, \citep{Africa2011} establishes
that for certain range of parameters, the solutions of the flocculation
model do blow up in finite time. To the best of our knowledge, for
the case $\overline{x}<\infty$ the long-term behavior of this model
has not been considered. Hence, our main goal in this paper is to
rigorously investigate the long-term behavior of the broad class of
flocculation models described in (\ref{eq: agg and growth model}).

When the long-term behavior of biological populations is considered,
many populations converge to a stable time-independent state. Thus,
identifying conditions under which a population converges to a stationary
state is one of the most important applications of mathematical population
modeling. Hence, our main goal in this work is to prove existence
of positive steady states of the microbial flocculation model (\ref{eq: agg and growth model}).
It is trivially true that a zero stationary solution of the microbial
flocculation model exists, but we are also interested in non-trivial
stationary solutions of the microbial flocculation model. Consequently,
in Section \ref{sec:Existence-of-a postive} we first show that under
some suitable conditions on the model parameters the equation (\ref{eq: agg and growth model})
has at least one non-trivial (non-zero and non-negative) stationary
solution. In Section \ref{sec:Numerical-results}, we present a numerical
scheme to approximate these stationary solutions. The numerical scheme
is based on spectral collocation method, and thus yields very accurate
results even for small approximation dimensions. Furthermore, we explore
the stationary solutions of the model for various biologically relevant
parameters in Section \ref{sec:Numerical Exploration of Steady States}
and give valuable insights for the efficient removal of suspended
particles.

\section{\label{sec:Existence-of-a postive}Existence of a positive stationary
solution}

The flocculation model under our consideration (\ref{eq: agg and growth model}),
accounts for physical mechanisms such as growth, removal, fragmentation,
aggregation and renewal of microbial flocs. Thus, under some conditions,
which balance these mechanisms, one could reasonably expect that the
model possesses a non-trivial stationary solution. Hence, our main
goal in this section is to derive sufficient conditions for the model
terms such that the equation (\ref{eq: agg and growth model}) engenders
a positive stationary solution.

The flocculation model in this form (\ref{eq: agg and growth model})
was first considered by Banasiak and Lamb in \citep{banasiak2009coagulation},
where they employed the flocculation model to describe the dynamical
behavior of phytoplankton cells. The authors showed that under some
conditions the flocculation model is well-posed, i.e., there exist
a unique, global in time, positive solution for every absolutely integrable
initial distribution. For the remainder of this work, we make the
following assumptions on the model rates for which well-posedness
of the solutions of the microbial flocculation equations has been
established by \citet{banasiak2009coagulation}:
\begin{alignat*}{1}
(\mathbf{A}1)\qquad & g\in C^{1}(I)\qquad g(x)>0\:\text{ for }x\in I=[0,\,\overline{x}]\\
(\mathbf{A}2)\qquad & k_{a}\in L^{\infty}(I\times I),\quad k_{a}(x,\,y)=k_{a}(y,\,x)\\
 & \text{ and }k_{a}(x,\,y)=0\,\,\,\text{ if }x+y\ge\overline{x}\,,\\
(\mathbf{A}3)\qquad & \mu\in C(I)\qquad\text{and }\mu\ge0\text{ a.e. on }I\,,\\
(\mathbf{A}4)\qquad & q\in L^{\infty}(I)\qquad\text{and }q\ge0\text{ a.e. on }I\,,\\
(\mathbf{A}5)\qquad & k_{f}\in C(I)\qquad k_{f}(0)=0\text{ and }k_{f}\ge0\text{ a.e. on }I\,,\\
(\mathbf{A}6)\qquad & \Gamma(\cdot,\,y)\in L^{\infty}(I),\quad\Gamma(x;\,y)\ge0\text{ for }x\in(0,\,y];\\
 & \text{and }\Gamma(x;\,y)=0\text{ for }x\in(y,\,\overline{x})\,.
\end{alignat*}

Assumption ($\mathbf{A}1$) states that the floc of any size has strictly
positive growth rate. This in turn implies that flocs can grow beyond
the maximal size $\overline{x}$, i.e., the model ignores what happens
beyond the maximal size $\overline{x}$ (as many authors in the literature
have done \citep{Ackleh1997,farkas2007stability,ackleh1997modeling}).
We also note that although the Assumption $(\mathbf{A}1)$ is widely
used in the literature it does generate biologically unrealistic condition
$g(0)>0$, i.e., the flocs of size zero also have positive growth
rate. However, this assumption is crucial for our work, and thus we
postpone the analysis of the case $g(0)=0$ for our future research.
Assumption $(\mathbf{A}2)$ states that for the aggregates of size
$x$ and $y$ the aggregation rate is zero if the combined size of
the aggregates is larger than the maximal size. Lastly, Assumption
$(\mathbf{A}3)$ on $\mu(x)$ enforces continuous dependence of the
removal on the size of a floc and ensures that every floc is removed
with a non-negative rate.

Recall that at a steady state we should have
\begin{equation}
u_{t}=0=\mathcal{F}[u]\,.\label{eq:steady state equation-1}
\end{equation}
By Assumption $(\mathbf{A}1)$, we know that $1/g\in C(I)$ and thus
we can define $u=f/g$ for some $f\in C(I)$. The substitution of
this $f$ into (\ref{eq:steady state equation-1}), integration between
$0$ and an arbitrary $x$, and rearrangment of the terms yields
\begin{alignat}{1}
f(x) & =f(0)-\int_{0}^{x}\frac{k_{f}(y)/2+\mu(y)}{g(y)}f(y)\,dy+\int_{0}^{x}\int_{z}^{\overline{x}}\frac{\Gamma(z;\,y)k_{f}(y)}{g(y)}f(y)\,dy\,dz\nonumber \\
 & \quad+\frac{1}{2}\int_{0}^{x}\int_{0}^{z}\frac{k_{a}(z-y,\,y)}{g(z-y)g(y)}f(z-y)f(y)\,dy\,dz-\int_{0}^{x}\frac{f(z)}{g(z)}\int_{0}^{\overline{x}}\frac{k_{a}(z,\,y)}{g(y)}f(y)\,dy\,dz\,.\label{eq: solving for f}
\end{alignat}
At this point we set 
\begin{equation}
f(0)=1=g(0)p(0)=\int_{0}^{\overline{x}}q(y)u(y)\,dy\,.\label{eq: condition on f}
\end{equation}
Note that if there is a function $f$ satisfying the equations (\ref{eq: solving for f})
and (\ref{eq: condition on f}), then $u_{*}$ is a steady state for
the modified choice of $q(x)$ 
\begin{equation}
q(x):=\frac{q(x)}{\int_{0}^{\overline{x}}\frac{q(y)f(y)}{g(y)}\,dy}\,.\label{eq: assumption on renewal}
\end{equation}
We now define the operator $\Phi$ as
\begin{alignat}{1}
\Phi[f](x) & :=1-\int_{0}^{x}\frac{k_{f}(y)/2+\mu(y)}{g(y)}f(y)\,dy+\int_{0}^{x}\int_{z}^{\overline{x}}\frac{\Gamma(z;\,y)k_{f}(y)}{g(y)}f(y)\,dy\,dz\nonumber \\
 & \quad+\frac{1}{2}\int_{0}^{x}\int_{0}^{z}\frac{k_{a}(z-y,\,y)}{g(z-y)g(y)}f(z-y)f(y)\,dy\,dz-\int_{0}^{x}\frac{f(z)}{g(z)}\int_{0}^{\overline{x}}\frac{k_{a}(z,\,y)}{g(y)}f(y)\,dy\,dz\,.\label{eq:definition of phi}
\end{alignat}
and will use a fixed point theorem to prove the existence of a fixed
point $f$ of $\Phi$ . This in turn will allow us to claim that equation
(\ref{eq:steady state equation-1}) has at least one non-trivial positive
solution.

The use of fixed point theorems for showing existence of non-trivial
stationary solutions is not new in size-structured population modeling.
For example, fixed point theorems, based on Leray-Schauder degree
theory, have been used to find stationary solutions of linear Sinko-Streifer
type equations \citep{Pruss1983,farkas2012steadystates}. Moreover,
the Schauder fixed point theorem has been used to establish the existence
of steady state solutions of nonlinear coagulation-fragmentation equations
\citep{Laurencot2005}. For our purposes we will use the Contraction
Mapping Theorem.

We carry out the analysis of this work on the space of continuous
functions $\mathcal{X}=C(I)$ with usual uniform (supremum) norm $\left\Vert \cdot\right\Vert _{u}$.
We also denote the usual essential supremum of a function by $\left\Vert \cdot\right\Vert _{\infty}$.
Since the positive cone in $C(I)$, denoted by $\left(C(I)\right)_{+}$,
is closed and convex, we choose $K$ to be $\left(C(I)\right)_{+}$.
Then $K_{r}=\overline{K\cap B_{r}(0)}$, where $B_{r}(0)\subset\mathcal{X}$
is an open ball of radius $r$ and centered at zero, and $r$ has
yet to be chosen. Note that since $K_{r}$ is also Banach space since
it is closed subspace of $\mathcal{X}$. Next, we show that one can
choose model rates such that the operator $\Phi$ defined in (\ref{eq:definition of phi})
maps $K_{r}$ to $K_{r}$ and is also contraction. This in turn implies
existence of a positive stationary solution of the operator $\mathcal{F}$.
We are now in a position to state the main result of this section
in the following theorem.
\begin{thm}
\label{thm:positive steady state existence}Assume that the condition
\[
0\le\frac{1}{2}k_{f}(x)-\mu(x),\tag{\textbf{C}1}
\]
holds true for all $x\in I$. For sufficiently small choice of $\left\Vert \frac{1}{g}\right\Vert _{1}$
the operator $\Phi$ defined in (\ref{eq:definition of phi}) has
a unique non-zero fixed point, $f_{*}\in K$ satisfying 
\begin{equation}
1\le\left\Vert f_{*}\right\Vert _{u}\le r\label{eq:bounds for the fixed point}
\end{equation}
for some $r\ge1$. Moreover, for the modified choice of the renewal
rate in (\ref{eq: assumption on renewal}), the non-zero and non-negative
function 
\begin{equation}
u_{*}=\frac{f_{*}}{g}\in C(I)\label{eq:stationary solution}
\end{equation}
is a unique stationary solution of the flocculation model defined
in (\ref{eq: agg and growth model}) on $K$.
\end{thm}

\begin{proof}
For $f\in K_{r}$ we have 
\begin{alignat*}{1}
\Phi[f] & \ge1-\int_{0}^{x}\frac{k_{f}(y)/2+\mu(y)}{g(y)}f(y)\,dy\\
 & \quad+\int_{0}^{x}\int_{z}^{x}\frac{\Gamma(z;\,y)k_{f}(y)}{g(y)}f(y)\,dy\,dz-\int_{0}^{x}\frac{f(z)}{g(z)}\int_{0}^{\overline{x}}\frac{k_{a}(z,\,y)}{g(y)}f(y)\,dy\,dz\\
 & \ge1-\int_{0}^{x}\frac{f(z)}{g(z)}(k_{f}(z)/2+\mu(z))\,dz+\int_{0}^{x}\frac{k_{f}(y)f(y)}{g(y)}\underbrace{\int_{0}^{y}\Gamma(z;\,y)\,dz}_{=1}\,dy-\left\Vert f\right\Vert _{u}^{2}\left\Vert k_{a}\right\Vert _{\infty}\left\Vert \frac{1}{g}\right\Vert _{1}^{2}\\
 & \ge\int_{0}^{x}\frac{f(z)}{g(z)}\left(\frac{1}{2}k_{f}(z)-\mu(z)\right)\,dz\,+1-r^{2}\cdot\left\Vert k_{a}\right\Vert _{\infty}\left\Vert \frac{1}{g}\right\Vert _{1}^{2},
\end{alignat*}
where $\left\Vert \cdot\right\Vert _{1}$ represents the usual $L^{1}$
norm on $I$. The first condition of the theorem $(\mathbf{C}1)$
guarantees that 
\[
\frac{1}{2}k_{f}(z)-\mu(z)>0\text{ for all }z\in I\,,
\]
so we can choose 
\begin{equation}
r=\left\Vert \frac{1}{g}\right\Vert _{1}^{-1}\left\Vert k_{a}\right\Vert _{\infty}^{-1/2}\label{eq:choice for r}
\end{equation}
 in $K_{r}$ such that $\Phi[f]\ge0$, i.e., $\Phi\,:\,K_{r}\to K$.
On the other hand, using the assumptions $(\mathbf{A}1)$-$(\mathbf{A}6)$,
it is straightforward to show that $\Phi(K_{r})\subset C(I)$. 

Next we prove that the operator $\Phi$ maps $K_{r}$ to $K_{r}$.
Consequently, for $f\in K_{r}$ it follows that 
\begin{alignat*}{1}
0\le\Phi[f](x) & \le1-\int_{0}^{x}\frac{k_{f}(y)/2+\mu(y)}{g(y)}f(y)\,dy+\int_{0}^{\overline{x}}\int_{z}^{\overline{x}}\frac{\Gamma(z;\,y)k_{f}(y)}{g(y)}f(y)\,dy\,dz\\
 & \quad-\int_{x}^{\overline{x}}\int_{z}^{\overline{x}}\frac{\Gamma(z;\,y)k_{f}(y)}{g(y)}f(y)\,dy\,dz+\frac{1}{2}\int_{0}^{x}\int_{0}^{z}\frac{k_{a}(z-y,\,y)}{g(z-y)g(y)}f(z-y)f(y)\,dy\,dz\\
 & \le1+\int_{x}^{\overline{x}}\frac{1}{g(y)}k_{f}(y)f(y)\,dy+\int_{0}^{x}\frac{1}{g(y)}\left(\frac{1}{2}k_{f}(y)-\mu(y)\right)f(y)\,dy\\
 & \quad+\frac{1}{2}\left\Vert f\right\Vert _{u}^{2}\cdot\left\Vert k_{a}\right\Vert _{\infty}\cdot\left\Vert \frac{1}{g}\right\Vert _{1}^{2}\\
 & \le1+\left\Vert f\right\Vert _{u}\left\Vert \frac{1}{g}\right\Vert _{1}\left[\left\Vert k_{f}\right\Vert _{u}+\left\Vert \frac{1}{2}k_{f}-\mu\right\Vert _{u}+r\left\Vert k_{a}\right\Vert _{\infty}\left\Vert \frac{1}{g}\right\Vert _{1}\right]\\
 & \le1+r\left\Vert \frac{1}{g}\right\Vert _{1}\left[\left\Vert k_{f}\right\Vert _{u}+\left\Vert \frac{1}{2}k_{f}-\mu\right\Vert _{u}+\left\Vert k_{a}\right\Vert _{\infty}^{1/2}\right]\,.
\end{alignat*}
At this point choosing $\left\Vert \frac{1}{g}\right\Vert _{1}$ sufficiently
small yields $r>1$ in (\ref{eq:choice for r}), and thus we can guarantee
that 
\[
\left\Vert \Phi[f]\right\Vert \le r
\]
for all $f\in K_{r}$. Hence the operator $\Phi$ maps $K_{r}$ to
$K_{r}$. 

Next we will prove that the operator $\Phi$ is in fact a contraction
mapping, i.e., for all $f,\,h\in K_{r}$
\[
\left\Vert \Phi[f]-\Phi[h]\right\Vert _{u}\le c\left\Vert f-h\right\Vert _{u}
\]
for some $c\in[0,1)$. 
\begin{alignat*}{1}
\left|\Phi[f](x)-\Phi[h](x)\right| & \le\int_{0}^{x}\frac{k_{f}(y)/2-\mu(y)}{g(y)}\left|f(y)-h(y)\right|\,dy+\int_{x}^{\overline{x}}\frac{k_{f}(y)}{g(y)}\left|f(y)-h(y)\right|\,dy\\
 & +\frac{1}{2}\int_{0}^{x}\int_{0}^{z}\frac{k_{a}(z-y,\,y)}{g(z-y)g(y)}f(z-y)\left|f(y)-h(y)\right|\,dy\,dz\\
 & +\frac{1}{2}\int_{0}^{x}\int_{0}^{z}\frac{k_{a}(z-y,\,y)}{g(z-y)g(y)}h(y)\left|f(z-y)-h(z-y)\right|\,dy\,dz\\
 & +\int_{0}^{x}\frac{f(z)}{g(z)}\int_{0}^{\overline{x}}\frac{k_{a}(z,\,y)}{g(y)}\left|f(y)-h(y)\right|\,dy\,dz\\
 & +\int_{0}^{x}\frac{\left|f(z)-h(z)\right|}{g(z)}\int_{0}^{\overline{x}}\frac{k_{a}(z,\,y)}{g(y)}h(y)\,dy\,dz\,.
\end{alignat*}
Taking the supremum of both sides at this point yields
\[
\left\Vert \Phi[f]-\Phi[h]\right\Vert _{u}\le\left\Vert f-h\right\Vert _{u}\left\Vert \frac{1}{g}\right\Vert _{1}\left[\left\Vert k_{f}\right\Vert _{u}+\left\Vert \frac{1}{2}k_{f}-\mu\right\Vert _{u}+\frac{3}{2}\left\Vert k_{a}\right\Vert _{\infty}^{1/2}\right]\,.
\]
Once again by choosing $\left\Vert \frac{1}{g}\right\Vert _{1}$ sufficiently
small, we can guarantee that the constant 
\[
c=\left\Vert \frac{1}{g}\right\Vert _{1}\left[\left\Vert k_{f}\right\Vert _{u}+\left\Vert \frac{1}{2}k_{f}-\mu\right\Vert _{u}+\frac{3}{2}\left\Vert k_{a}\right\Vert _{\infty}^{1/2}\right]
\]
is less than $1$. This in turn implies that the operator $\Phi$
defined in (\ref{eq:definition of phi}) is also a contraction mapping.
Hence, the Contraction Mapping Theorem guarantees the existence of
a unique positive fixed point of $\Phi$ satisfying the bounds (\ref{eq:bounds for the fixed point}).
Therefore, the function $u_{*}=f_{*}/g$ is a stationary solution
of the flocculation equations (\ref{eq: agg and growth model}). Moreover,
from the assumption ($\mathbf{A}1$) and the continuity of the fixed
point $f_{*}$ it follows that $u_{*}$ is non-zero, non-negative
and continuous on $I$.
\end{proof}

\section{\label{sec:Numerical-results}Numerical approximation of non-trivial
stationary solutions}

Recall that at steady state the microbial flocculation equations reduce
to first order integro-differential equation 
\begin{alignat}{1}
\partial_{x}(gu_{*}) & =\frac{1}{2}\int_{0}^{x}k_{a}(x-y,\,y)u_{*}(x-y)u_{*}(y)\,dy-u_{*}(x)\int_{0}^{\overline{x}-x}k_{a}(x,\,y)u_{*}(y)\,dy\nonumber \\
 & \quad+\int_{x}^{\overline{x}}\Gamma(x;\,y)k_{f}(y)u_{*}(y)\,dy-\frac{1}{2}k_{f}(x)u_{*}(x)-\mu(x)u_{*}(x)\label{eq:steady-state IVP}
\end{alignat}
with a boundary condition

\[
g(0)u_{*}(0)=\int_{0}^{\overline{x}}q(x)u_{*}(x)dx\,.
\]
This in turn can be expressed in the form of a boundary value problem.
\begin{equation}
\frac{du_{*}}{dx}=F(x,\,u_{*}),\qquad u_{*}(0)=\int_{0}^{\overline{x}}q(x)u_{*}(x)dx/g(0)\,.\label{eq:reduced boundary VP}
\end{equation}
If we set 
\[
\int_{0}^{\overline{x}}q(x)u_{*}(x)dx=1
\]
and adjust the renewal rate to the obtained steady state as in Section
\ref{sec:Existence-of-a postive}, we get an initial value problem
\begin{equation}
\frac{du_{*}}{dx}=F(x,\,u_{*}),\qquad u_{*}(0)=1/g(0)\,.\label{eq:reduced IVP}
\end{equation}
Using Picard's Existence Theorem for IVPs, one can show that the IVP
(\ref{eq:reduced IVP}) has a unique solution for a suitably chosen
interval around the initial condition. Therefore, the above IVP is
well-posed. Note that we still need results of Section \ref{sec:Existence-of-a postive}
for the existence of a positive stationary solution since Picard's
Existence Theorem does not guarantee positivity of solutions of (\ref{eq:reduced IVP}). 

The solutions of the IVP (\ref{eq:reduced IVP}) can be approximated
using various numerical schemes such as finite difference, finite
element, and spectral methods \citep{nicmanis1998finiteelement,matveev2015afast,mirzaev2017anumerical}.
Towards this end, in this section, we develop a numerical scheme to
approximate the solutions of the IVP (\ref{eq:reduced IVP}). The
numerical scheme is based on spectral collocation method and thus
yields very accurate results even for small approximation dimensions. 

 Spectral collocation method, also known as pseudospectral methods,
is a subclass of Galerkin spectral methods \citep{fornberg1998apractical,trefethen2000spectral}.
Spectral methods, in general, have higher accuracy compared to Finite
difference and Finite Element methods and thus have widespread use
for the numerical simulation of partial differential equations. The
main idea behind the method is to express numerical solutions as a
finite expansion of some set of basis functions on a number of points
on the domain (i.e., collocation points). Convergence of the approximations
depend only on the smoothness of the solutions and thus the method
can attain high precision even with a few grid points.

We approximate the stationary solution $u_{*}(x)$ as linear combination
of $N$th degree polynomials $\left\{ \phi_{j}(x)\right\} _{j=0}^{N}$,
i.e.,
\[
\left[I_{N}u_{*}\right](x)=\sum_{j=0}^{N}u_{*}(x_{j})\phi_{j}(x)\,,
\]
for some collocation points $\left\{ x_{j}\right\} _{j=0}^{N}\subset I$.
Furthermore, we require that this approximation is exact at collocation
points, i.e.,
\[
\phi_{j}(x_{k})=\begin{cases}
1 & \text{if }j=k\\
0\, & \text{otherwise}
\end{cases}\,.
\]
The polynomials $\left\{ \phi_{j}(x)\right\} _{j=0}^{N}$ satisfying
the above requirements are called cardinal functions. Straightforward
way to compute such polynomials is using Lagrangian interpolation,
i.e., 
\begin{equation}
\phi_{j}(x)=\frac{\pi(x)}{\pi'(x_{j})(x-x_{j})}\,,\label{eq:cardinal functions}
\end{equation}
where $\pi(x)=\Pi_{j=0}^{N}(x-x_{j})$. Consequently, derivative can
be computed as 
\[
\frac{d}{dx}\left[I_{N}u_{*}\right](x)=\sum_{j=0}^{N}u_{*}(x_{j})\phi_{j}'(x)\,.
\]
This in turn implies that derivative at $x_{i}$ can be computed as
linear combination of $\left\{ u_{*}(x_{j})\right\} _{j=0}^{N}$.
Plugging in collocation points into above equation yields the differentiation
matrix $[D]_{i,\,j}=\phi_{j}'(x_{i})$. Entries of $D$ can be computed
explicitly as
\[
D_{ij}=\frac{\pi'(x_{i})}{\pi'(x_{j})(x_{i}-x_{j})}
\]
 for off-diagonal entries $i\ne j$ and
\[
D_{jj}=\sum_{\begin{array}{c}
k=0,\,k\ne j\end{array}}^{N}\frac{1}{x_{j}-x_{k}}\,
\]
for diagonal entries. Let $\vec{u}$ denote the vector 
\[
\begin{bmatrix}u_{*}(x_{0}) & \cdots & u_{*}(x_{N})\end{bmatrix}^{T}\,,
\]
then derivatives at collocation points are given by
\[
\begin{bmatrix}u_{*}'(x) & \cdots & u_{*}'(x_{N})\end{bmatrix}^{T}=D\vec{u}\,.
\]

When a smooth function is interpolated by polynomials in $N$ equally
space points, the approximations sometimes fail to converge as $N\to\infty$,
which is also known as \emph{Runge phenomenon.} Moreover, when using
uniform gird points, the elements of the differentiation matrix not
only fail to converge but they get worse and diverge as $N\to\infty$.
For the spectral collocation methods, it is a general consensus to
cluster the grid points roughly quadratically toward the endpoints
of the interval \citep{fornberg1998apractical}. Therefore, for collocation
points we use unevenly space grid points. We employ non-uniform shifted
Chebyshev-Gauss-Lobatto grid points,
\begin{equation}
x_{j}=\left(1-\cos\left(\frac{j\pi}{N}\right)\right)\frac{\overline{x}}{2}\qquad\forall i=0,1,\dots,N\,.\label{eq: Chebyshev grids}
\end{equation}
Points $\left\{ \cos\left(\frac{j\pi}{N}\right)\right\} _{j=0}^{N}$
are in fact extrema of $N$th Chebyshev polynomial on the interval
$-1\le x\le1$.

\begin{figure}
\centering{}\subfloat[\label{fig:Linear-case}]{\centering{}\includegraphics[width=0.3\textwidth]{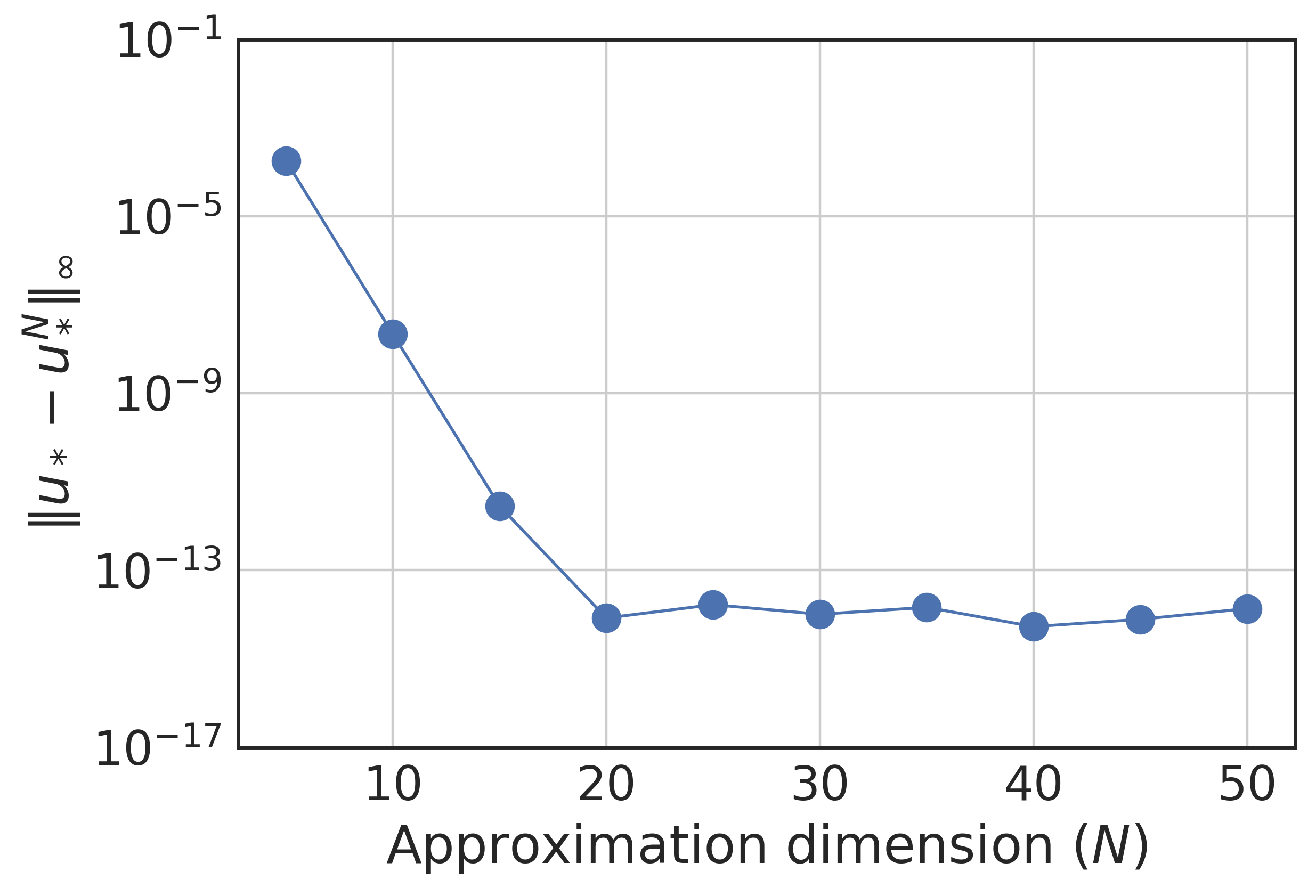}}~~~\subfloat[\label{fig:nonlinear case}]{\includegraphics[width=0.3\textwidth]{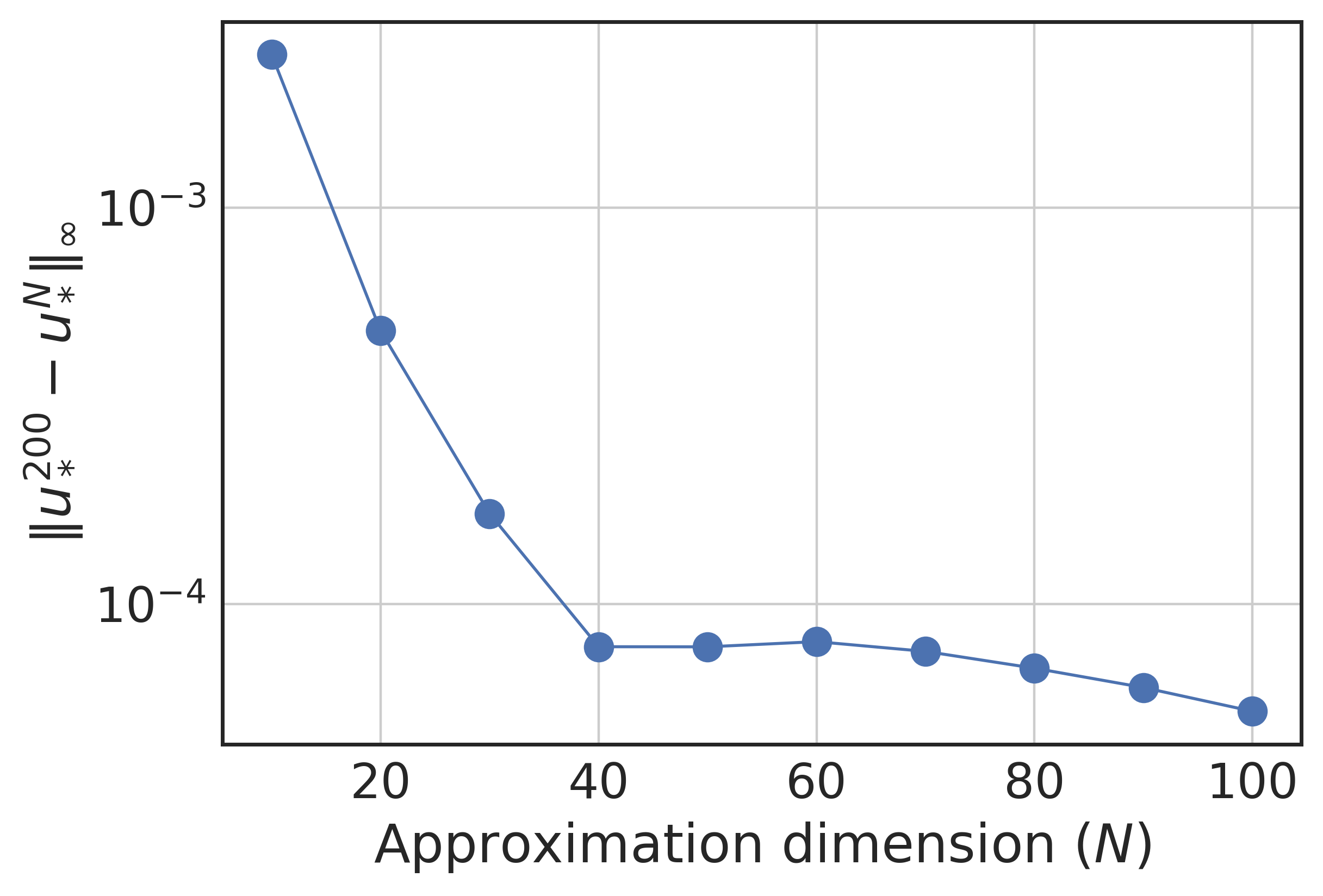}}~~~\subfloat[\label{fig:example solutions}]{\centering{}\includegraphics[width=0.3\textwidth]{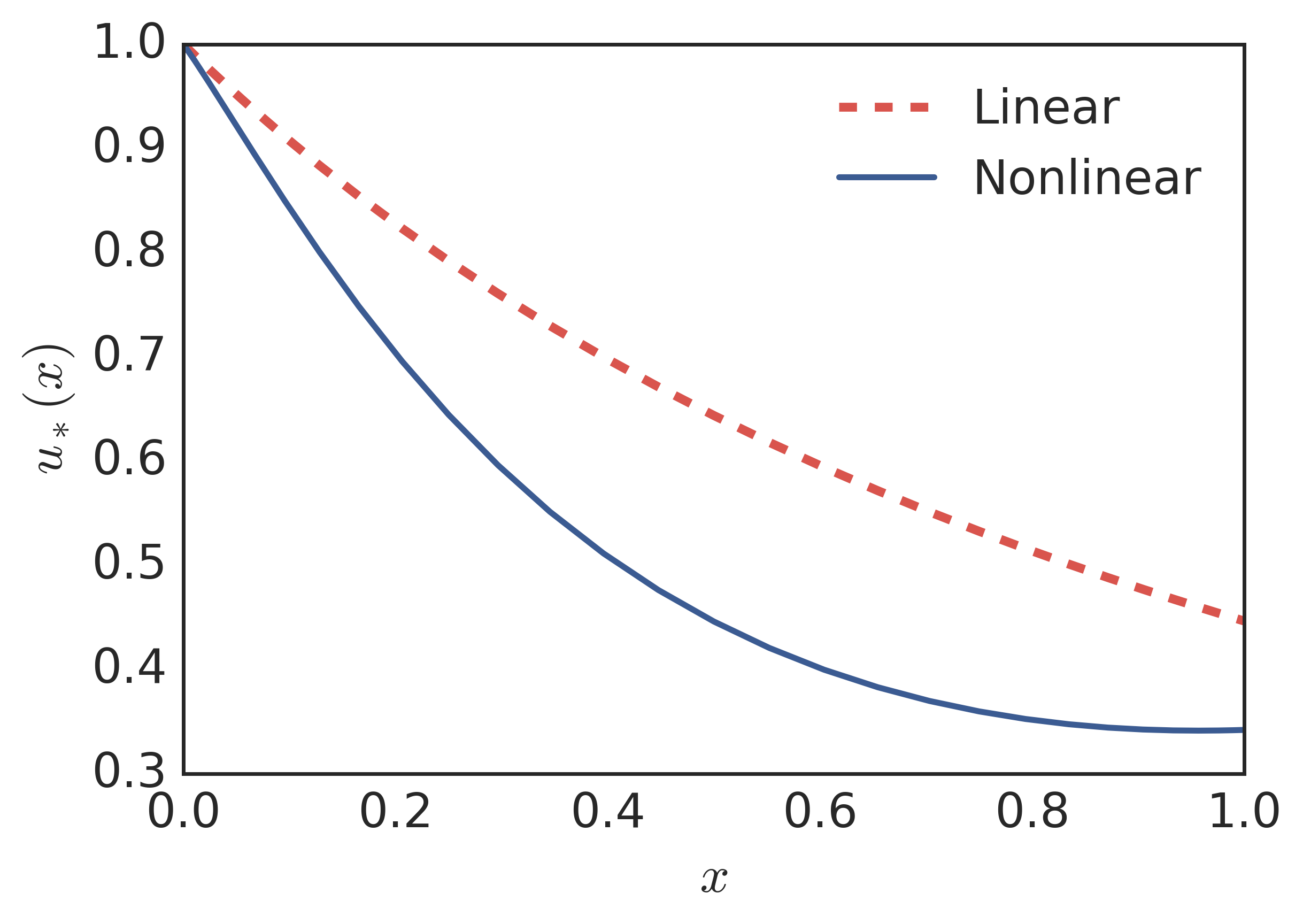}}\caption{\label{fig:Numerical-simulation-results-1}a) Error in approximation
of the steady states of linear Sinko-Streifer. Compared to existing
analytical solution. b) Error in approximation of the steady states
of nonlinear flocculation equations compared to approximate solution
at $N=200$. c) Example steady state solutions for linear Sinko-Streifer
equations and nonlinear microbial flocculation equations}
\end{figure}

For the numerical approximations of the steady state of the microbial
flocculation model (\ref{eq: agg and growth model})For integral approximations
we used Gaussian quadrature,
\[
\int_{0}^{\overline{x}}f(s)\,ds\approx\sum_{i=0}^{N}w_{i}f(x_{i})\,.
\]
We require Gaussian quadrature to be exact for chosen cardinal functions
$\left\{ \phi_{i}(x)\right\} _{i=0}^{N}$, which yields weights 
\[
w_{i}=\int_{0}^{\overline{x}}\phi_{i}(s)\,ds\,.
\]

Moreover, in the evaluation of the integrals in (\ref{eq:Aggregation})
one has to evaluate $\left[I_{N}u_{*}\right](x)$ at non-collocation
points, i.e.,
\begin{equation}
\left[I_{N}u_{*}\right](x_{k}-x_{i})=\sum_{j=0}^{N}u_{*}(x_{j})\phi_{j}(x_{k}-x_{i})\,.\label{eq:non collocation points}
\end{equation}
Once $\phi_{j}(x_{k}-x_{i})$ are calculated explicitly using (\ref{eq:cardinal functions}),
the approximations at non-collocation points (\ref{eq:non collocation points})
can be evaluated as linear combination of entries of $\vec{p}$. For
efficient implementation the elements $\phi_{j}(x_{k}-x_{i})$ can
be initialized as entries of three dimensional array, i.e.,
\[
[\Phi]_{k,i,j}=\begin{cases}
\phi_{j}(x_{k}-x_{i}) & \text{ if }k\ge i\\
0 & \text{otherwise}
\end{cases}\,.
\]
Consequently, approximation for (\ref{eq:non collocation points})
can be obtained as 
\[
\left[I_{N}u_{*}\right](x_{k}-x_{i})=\sum_{j=0}^{N}[\Phi]_{k,i,j}\vec{u}_{j}\,,
\]
which is simply dot product of three dimensional array $\Phi$ with
the vector $\vec{u}$.

To verify convergence of the proposed approximation scheme, we apply
first applied our numerical scheme to a special case of flocculation
equations for which an exact form of the stationary solution is available.
In particular, in the absence of aggregation and fragmentation the
flocculation equations simplify to linear Sinko-Streifer model \citep{sinko1967anew}.
The model describes the dynamics of single species populations and
takes into account the physiological characteristics of animals of
different sizes (and/or ages). 

Setting the right side of the equation (\ref{eq:Growth}) to zero
and integrating over the size on $(0,\,x)$ yields the exact stationary
solution
\begin{equation}
u_{*}(x)=\frac{1}{g(x)}\exp\left(-\int_{0}^{x}\frac{\mu(s)}{g(s)}\,ds\right)\,.\label{eq:exact stationary solution-1}
\end{equation}
In Figure \ref{fig:Linear-case}, we compared numerical approximations
of the stationary solution to exact solution given in (\ref{eq:exact stationary solution-1}).
The absolute error decreases exponentially fast for increasing approximation
dimension $N$. Numerical approximation attains machine precision
for $N\ge20$. As for the full nonlinear flocculation equations, to
the best of our knowledge, no analytical starionary solution is available.
Therefore, in Figure \ref{fig:nonlinear case}, we have plotted relative
error of approximations compared to a approximation at $N=200$. Relative
error decreases fast and achieves four digit precision for $N\ge40$.
Moreover, example steady state solutions for both linear and nonlinear
cases are illustrated in Figure \ref{fig:example solutions}. 

\section{Numerical Exploration of Steady States\label{sec:Numerical Exploration of Steady States}}

Having the approximation scheme in hand, in this section, we explore
the stationary solutions of the model for various biologically relevant
parameters and give valuable insights for the efficient removal of
suspended particles. For the purpose of illustration, the aggregation
kernel was chosen to describe a flow within laminar shear flow \citep{saffman1956onthe}
(i.e., \emph{orthokinetic} aggregation)
\[
k_{a}(x,\,y)=1.3\left(\frac{\epsilon}{\nu}\right)^{1/2}\left(x^{1/3}+y^{1/3}\right)^{3}\,,
\]
where $\epsilon$ represents the homogeneous turbulent energy dissipation
rate of the stirred tank and $\nu$ is the kinematic viscosity of
the suspending fluid. The quantity
\[
\dot{\gamma}:=\left(\frac{\epsilon}{\nu}\right)^{1/2}
\]
is often referred to as a ``volume average shear rate'' (hereby,
referred to as the ``shear rate'') of the stirring tank. 

We employ the fragmentation rate given by \citet{spicer1995shearinduced}
\[
k_{f}(x)=C_{f}x^{1/3}\,,
\]
where $C_{f}$ is the breakage rate coefficient for shear-induced
fragmentation. \citet{spicer1996coagulation} have experimentally
shown that there is a power law relation between the shear rate $\dot{\gamma}$
and the breakage rate $C_{f}$, 
\[
C_{f}=a\dot{\gamma}^{b}
\]
where $a$ and $b$ fitting parameters specific to a flow type. For
our purposes we use the parameters for laminar shear flow reported
by \citet{flesch1999laminar}, 
\[
a=7\times10^{-4},\qquad b=1.6\,.
\]

\begin{table}
\begin{centering}
\begin{tabular}{cccc}
\hline 
Parameter & Symbol & Value & Source\tabularnewline
\hline 
\hline 
Kinematic viscosity of water at $20^{\circ}$C & $\nu$ & $10^{-6}\,m^{2}/s$ & \citep{jewett2008physics}\tabularnewline
Shear rate & $\dot{\gamma}$ & $0-100\,s^{-1}$ & \citep{piani2014rheology}\tabularnewline
Fitting parameter & $a$ & $7\times10^{-4}$ & \citep{flesch1999laminar}\tabularnewline
Fitting parameter & $b$ & $1.6$ & \citep{flesch1999laminar}\tabularnewline
Removal rate & $C_{\mu}$ & $1/\dot{\gamma}$ & Assumed\tabularnewline
Growth rate & $C_{g}$ & $0-10$ & Assumed\tabularnewline
Renewal rate (surface erosion) & $C_{q}$ & $\left(\int_{0}^{1}(y+1)u_{*}(y)\,dy\right)^{-1}$ & This paper (\ref{eq: explicit value for C_q})\tabularnewline
\hline 
\end{tabular}
\par\end{centering}
\caption{\label{tab:Parameter-values-used}Model parameters and their values
used in simulations}
\end{table}
For a post-fragmentation density function we chose the well-known
Beta distribution\footnote{Although normal and log-normal distributions are mostly used in the
literature, Kightley et al. \citep{kightley2018fragmentation} have
provided evidence that the Beta density function describes the fragmentation
of small bacterial flocs. } with $\alpha=\beta=2$,
\[
\Gamma(x,\,y)=\mathbbm{1}_{[0,\,y]}(x)\frac{6x(y-x)}{y^{3}}\,,
\]
where $\mathbbm{1}_{I}$ is the indicator function on the interval
$I=[0,\,1]$. Removal rate is assumed to be linearly proportional
to the volume of the floc, 
\[
\mu(x)=C_{\mu}x\,.
\]
Since flocs sediment slower under large shear rates, the removal rate
should be inversely proportional to the shear rate of the stirring
tank. Therefore, for the remaining of the paper we set the removal
rate to 
\[
C_{\mu}=\exp\left(-\dot{\gamma}\right)\,.
\]
Renewal (or surface erosion rate) is assumed to be proportional to
the surface area of a floc, i.e., 
\[
q(x)=C_{q}x^{2/3}\,.
\]
where $C_{q}$ some positive real number. Finally, we chose growth
rate arbitrarily to fulfill positivity condition $\left(\text{\textbf{A}}1\right)$,
\[
g(x)=C_{g}(x+1)\,.
\]
Note that, once a stationary solution is found, the constant $C_{q}$
needs to be set to 
\begin{equation}
C_{q}:=\left(\int_{0}^{1}(y+1)p_{*}(y)\,dy\right)^{-1}\,.\label{eq: explicit value for C_q}
\end{equation}
The remaining parameters $C_{g}$ and $\dot{\gamma}$ can be set to
an arbitrary positive real number. 

\begin{figure}
\centering{}\subfloat[\label{fig:increasing growth rate-1}]{\includegraphics[width=0.4\textwidth]{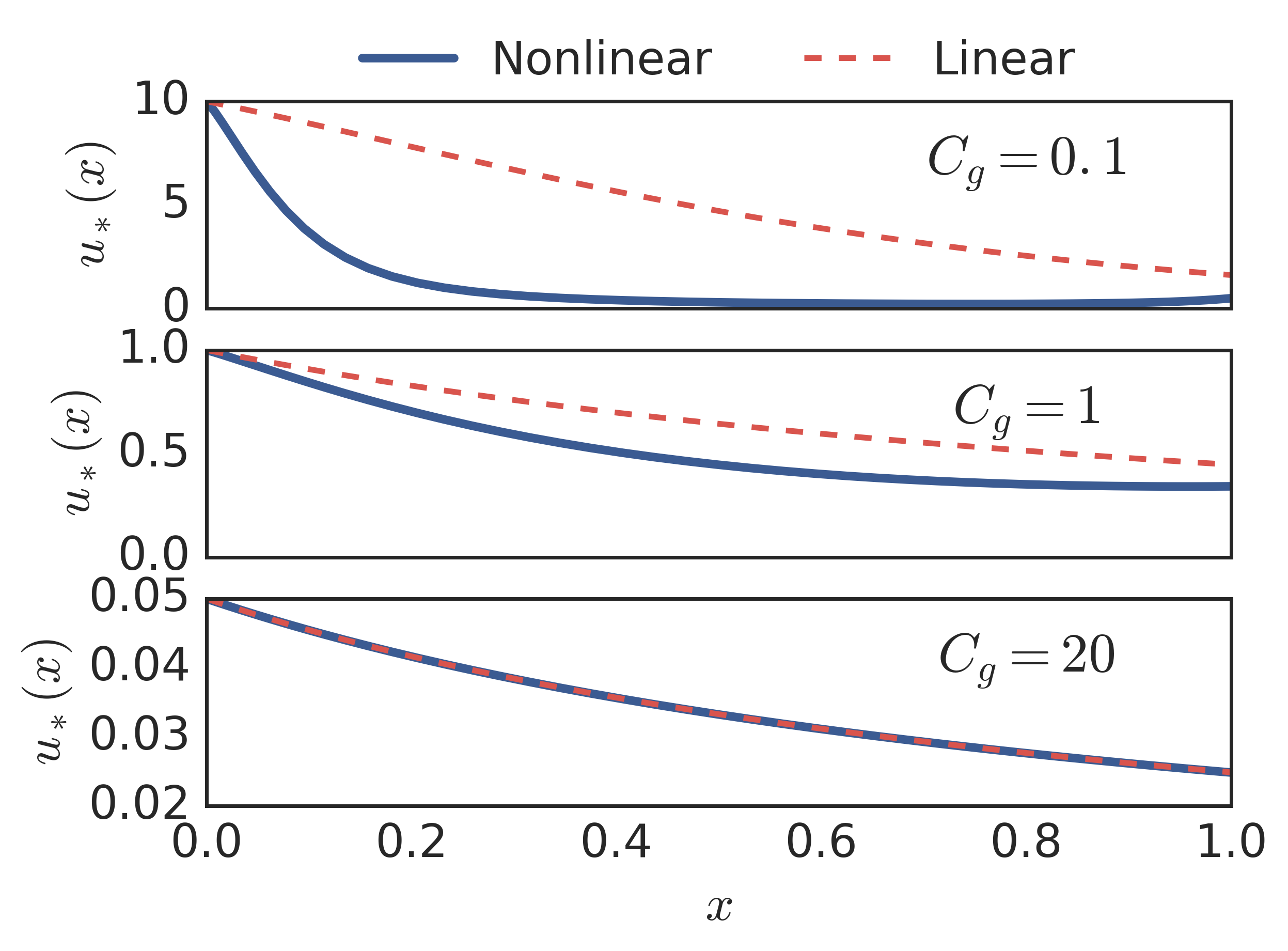}}~~~\subfloat[\label{fig:decreasing shear rate-1}]{\includegraphics[width=0.4\textwidth]{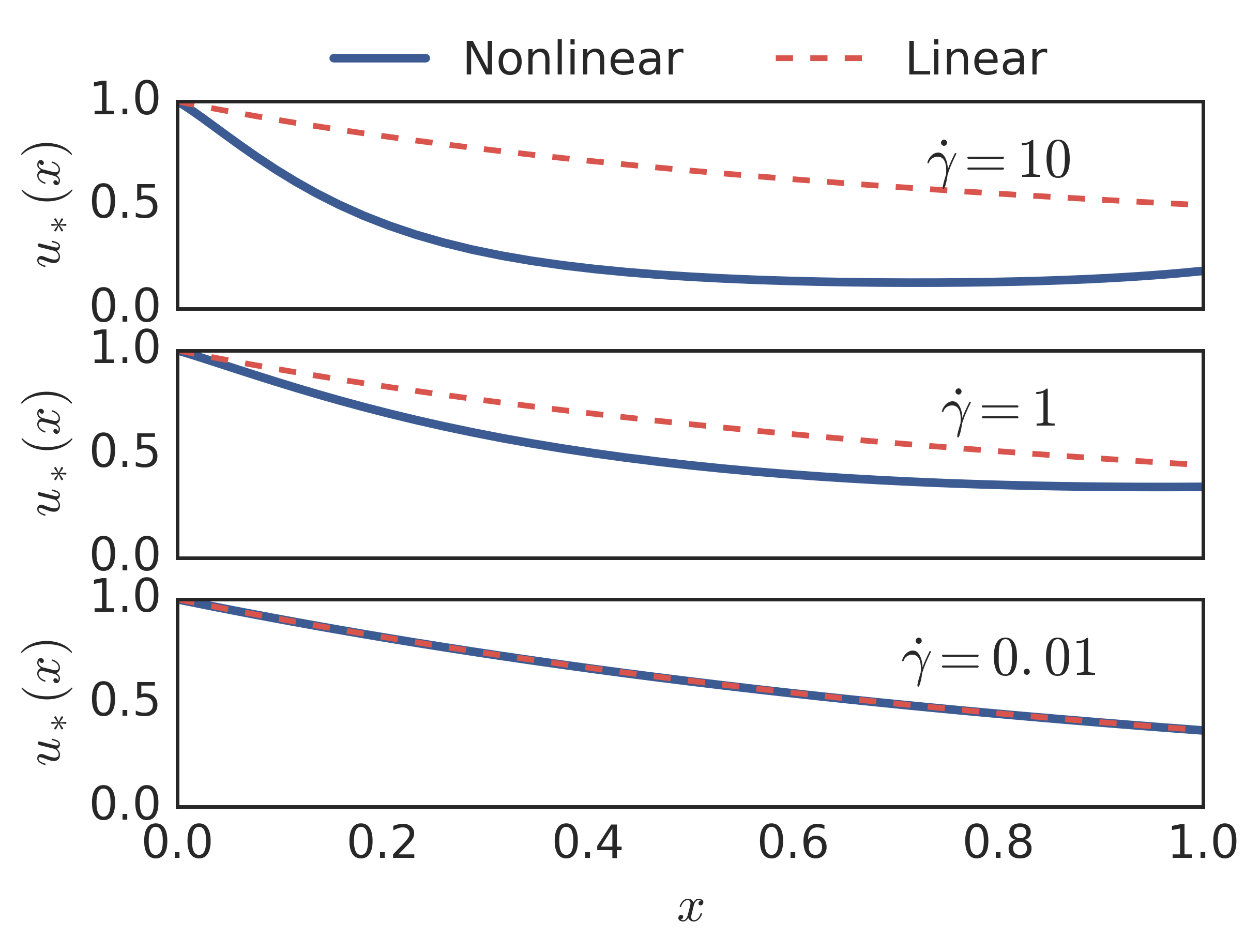}}\caption{Results of some numerical simulations. Dashed red and solid blue lines
correspond to the stationary of the linear Sinko-Streifer and the
nonlinear microbial flocculation equations, respectively. a) Steady
state solutions for increasing growth rate b) Steady state solutions
for decreasing shear rate }
\end{figure}

Stationary solutions of the linear and nonlinear case for several
different values of $C_{g}$ are depicted in Figure \ref{fig:increasing growth rate-1}.
It is interesting to note that for increasing values of $C_{g}$ the
growth dominates and thus stationary solution of the nonlinear case
approaches the stationary solutions of the linear case. Analogously,
as one could expect, results in Figure \ref{fig:decreasing shear rate-1}
indicate that as $\dot{\gamma}\to0$ the stationary solutions of the
microbial flocculation model converge to that of essentially linear
Sinko-Streifer model. 

\begin{figure}
\centering{}\subfloat[\label{fig:increasing shear rate-1}]{\includegraphics[width=0.3\textwidth]{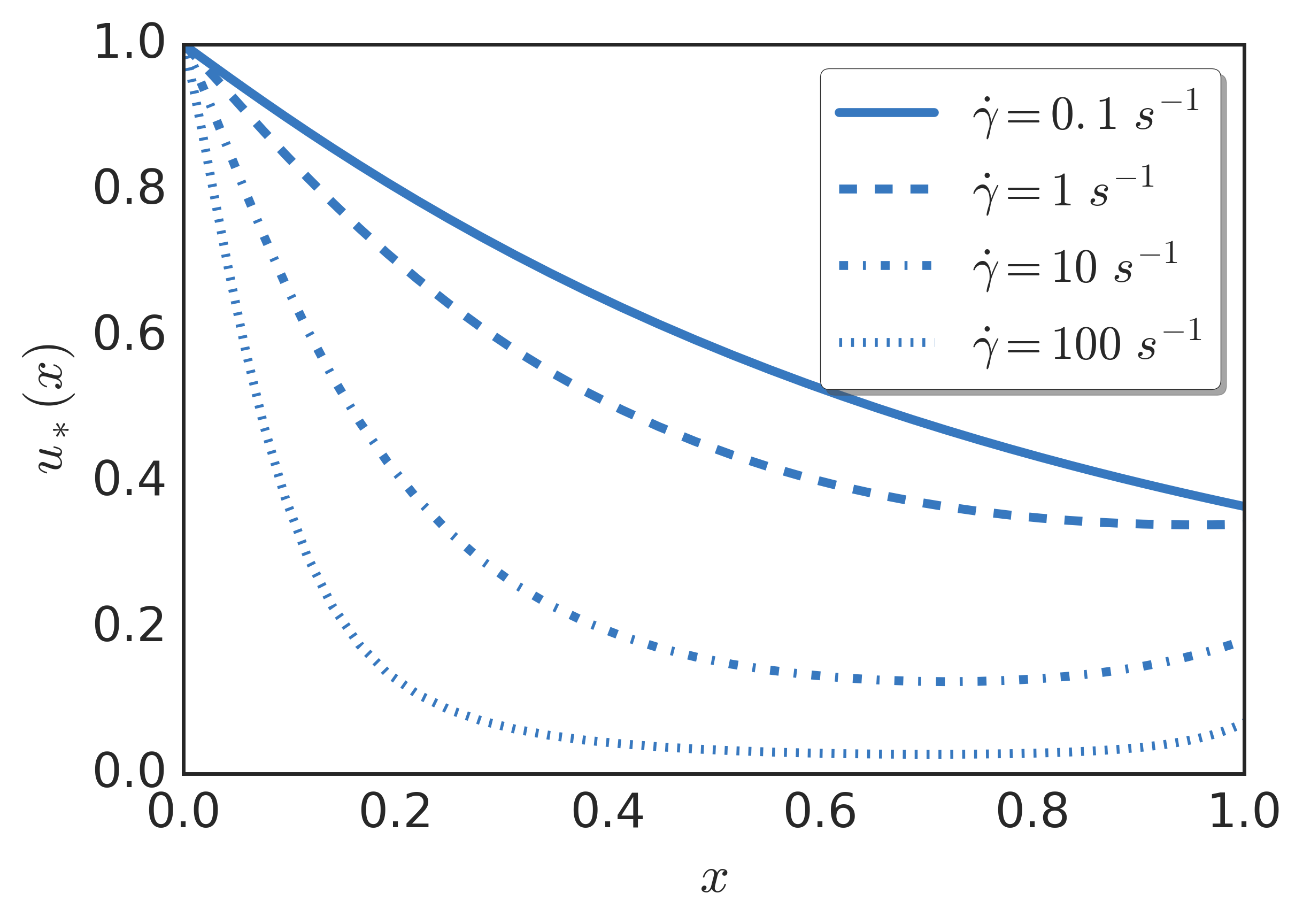}}~~~\subfloat[\label{fig:Growth-vs-shear rate,-1}]{\includegraphics[width=0.3\textwidth]{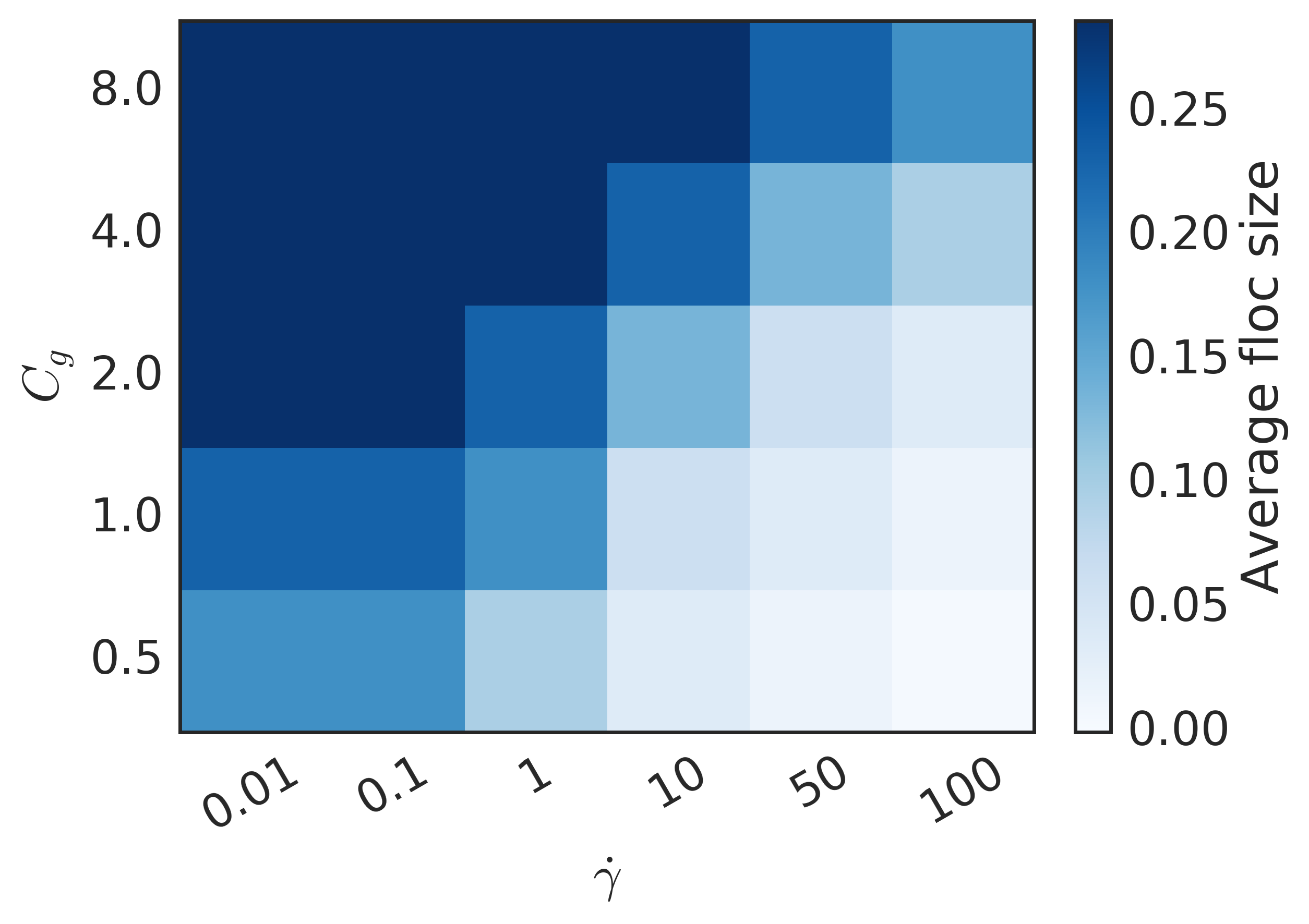}}\caption{\label{fig:Parameter-range-exploration:-1} Effect of the shear rate
on the average floc size and the renewal rate. a) Increasing the shear
rate results in stationary distributions with smaller average floc
size. b) For each given growth rate of a microbial floc, increasing
the shear rate decreases the average floc size. }
\end{figure}

Thorough control of floc formation is crucial for proper operation
of bioreactors used in fermentation industry and waste water treatment.
For efficient removal of suspended particles, it is usually desirable
to have larger and denser flocs that settle faster under gravitational
forces. Therefore, in Figures \ref{fig:increasing shear rate-1} and
\ref{fig:Growth-vs-shear rate,-1}, we investigated the effects of
the shear rate on the average floc size. One can observe in Figure
\ref{fig:increasing shear rate-1} that increasing the shear rate
results in increased fragmentation of the flocs and thus drives the
stationary distribution into the smaller size range, which is also
consistent with the results of \citep{flesch1999laminar}. Moreover,
results in Figure \ref{fig:Growth-vs-shear rate,-1} indicate that
for each given growth rate of a microbial floc one can adjust the
shear rate to yield an optimal average floc size.

As it is stated in Theorem \ref{thm:positive steady state existence},
for the existence of a positive stationary solution the renewal rate
has to be modified according to equation (\ref{eq: assumption on renewal}).
Consequently, in Figure \ref{fig:renewal rate surface}, for different
growth and shear rates we computed the renewal rate $C_{q}$ based
on the equation (\ref{eq: explicit value for C_q}). Computed renewal
rates lie on a smooth three-dimensional surface. Moreover, the results
indicate that the renewal rate is directly proportional to both growth
and shear rates. 

The numerical algorithm described in this chapter can be easily modified
to find steady states of the boundary value problem in (\ref{eq:reduced boundary VP}).
In this general case, the results of Section \ref{sec:Existence-of-a postive}
do not guarantee the existence of a positive steady state. However,
Figure \ref{fig:steady state for star mark} depicts that positive
steady states exist for the points not lying on the surface illustrated
in Figure \ref{fig:renewal rate surface}. This in turn suggests well-posedness
of the boundary value problem given in (\ref{eq:reduced boundary VP})
for sufficiently smooth model rates. Therefore, as the future plan,
we wish to further investigate well-posedness of this boundary value
problem.
\begin{center}
\begin{figure}
\centering{}\subfloat[\label{fig:renewal rate surface}]{\begin{centering}
\includegraphics[width=0.4\textwidth]{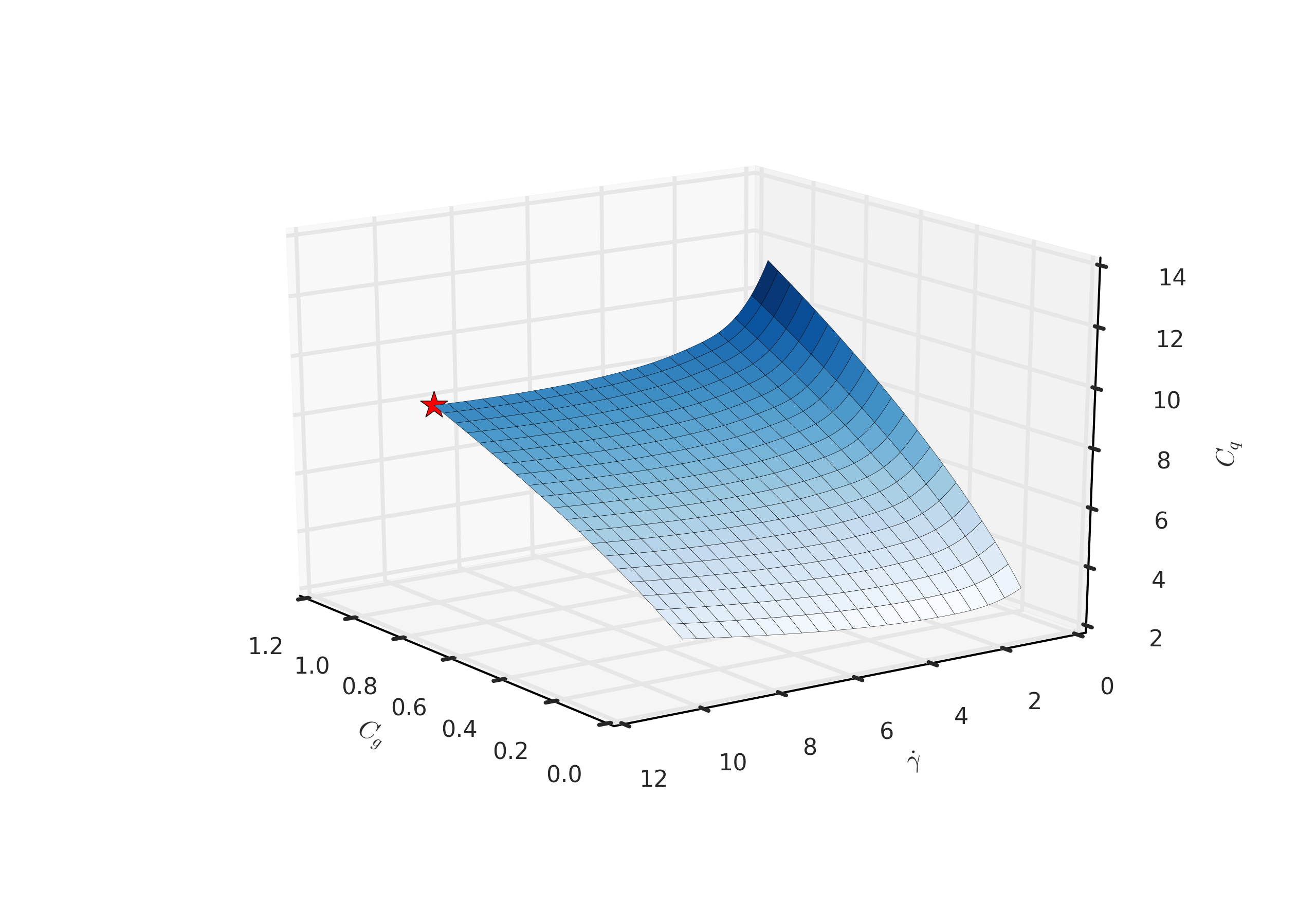}
\par\end{centering}
}~~~~\subfloat[\label{fig:steady state for star mark}]{\begin{centering}
\includegraphics[width=0.4\textwidth]{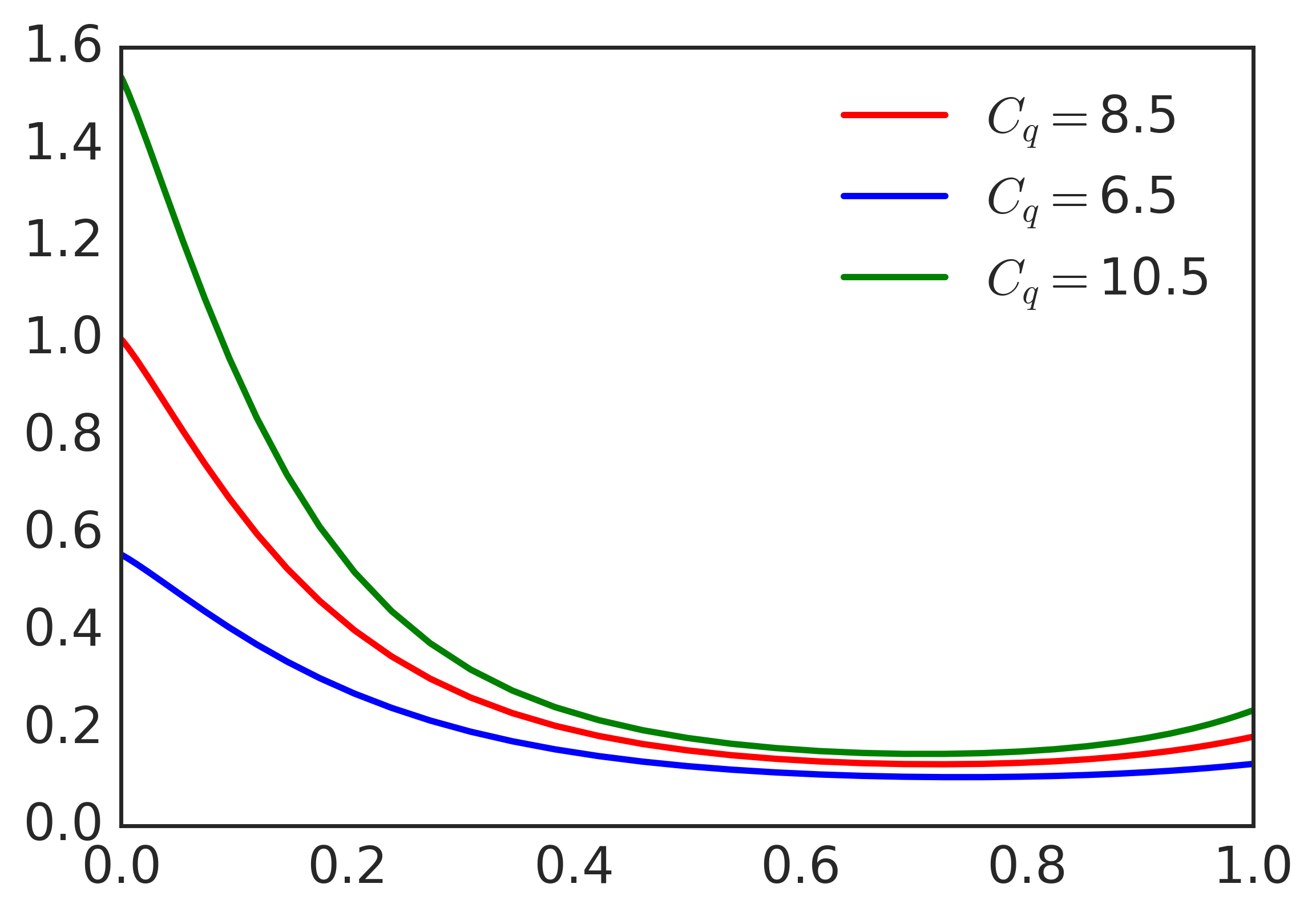}
\par\end{centering}
}\caption{Effect of growth and shear rates on the renewal rate, $C_{q}$. a)
Renewal rates form a smooth surface. Marked red star corresponds to
the point $C_{g}=1$, $\dot{\gamma}=1$ and $C_{q}=8.5$ b) Steady
states for marked red star and some points below and above the marked
point.}
\end{figure}
\par\end{center}

\section{\label{sec:Concluding-remarks}Concluding remarks}

Our primary motivation in this paper is to investigate the ultimate
behavior of solutions of a generalized size-structured flocculation
model. The model accounts for a broad range of biological phenomena
(necessary for survival of a community of microorganism in a suspension)
including growth, aggregation, fragmentation, removal due to predation,
and gravitational sedimentation. Moreover, the number of cells that
erode from a floc and enter the single cell population is modeled
with McKendrick-von Foerster type renewal boundary equation. Although
it has been shown that the model has a unique positive solution, to
the best of our knowledge, the large time behavior of those solutions
has not been studied. Therefore, in this paper we showed that under
relatively weak restrictions the flocculation model possesses a non-trivial
stationary solution (in addition to the trivial stationary solution). 

We developed a numerical scheme based on spectral collocation method
to approximate these nontrivial stationary solutions. Our numerical
exploration of the stationary solutions of the microbial flocculation
model indicate that for sufficiently large growth rates the stationary
solutions converge to stationary solutions of the linear Sinko-Streifer
model (\ref{eq:Growth}). Therefore, as future research, we plan to
study this behavior analytically. 

Furthermore, in Section \ref{sec:Numerical Exploration of Steady States},
we numerically investigated physically relevant parameter ranges.
In particular, we studied the effects of the shear rate on the average
floc size. Our results (consistent with the result in the literature
\citep{flesch1999laminar,Spicer1998,Spicer1997}) indicate that increasing
the shear rate of the stirring tank results in increased fragmentation,
and thus decreasing the steady-state average floc size.

\section*{Acknowledgments}

Funding for this research was supported in part by grants NSF-DMS
1225878 and NIH-NIGMS 2R01GM069438-06A2. 

\bibliographystyle{apalike}
\bibliography{C:/Users/marakanda/Dropbox/misc/texmf/bibtex/bib/my_zotero_CU}

\end{document}